\documentclass[12pt]{amsart}
\usepackage{amssymb,latexsym}
\usepackage{amsmath}
\usepackage{amsthm}
\usepackage[T1,T5]{fontenc}
\usepackage{enumerate}
\makeatletter
\@namedef{subjclassname@2010}{%
  \textup{2010} Mathematics Subject Classification}
\makeatother

\newtheorem{thm}{Theorem}[section]
\newtheorem{cor}[thm]{Corollary}
\newtheorem{lem}[thm]{Lemma}
\newtheorem{prop}[thm]{Proposition}

\newtheorem{mainthm}{Theorem}

\theoremstyle{definition}
\newtheorem{defin}[thm]{Definition}
\newtheorem{rem}[thm]{Remark}
\newtheorem{exa}[thm]{Example}

\numberwithin{equation}{section}

\newcommand{\abs}[1]{\left|#1\right|}

\newcommand{\C}{\mathbb{C}}

\newcommand{\Cn}{\mathbb{C}^n}

\newcommand{\F}{\mbox{$\mathcal{F}$}}

\newcommand{\PSH}[1]{\mbox{$\mathcal{PSH}(#1)$}}
\newcommand{\USC}[1]{\mbox{$\mathcal{USC}(#1)$}}

\newcommand{\PSHo}[1]{\mbox{$\mathcal{PSH}^o(#1)$}}

\renewcommand{\F}{\mbox{$\mathcal{F}$}}
\newcommand{\M}[3]{\mbox{$\mathcal{M}^{#2}_{#1}\left(#3\right)$}}

\newcommand{\D}{\mathbb{D}}
\newcommand{\J}[2]{\mathcal{J}_{#1}(#2)}

\renewcommand{\phi}{\varphi}

\renewcommand{\bar}[1]{\overline{#1}}

\newcommand{\supp}[1]{\operatorname{supp}{\!#1}}
\renewcommand{\tilde}[1]{\widetilde{#1}}

\renewcommand{\epsilon}{\varepsilon}

\begin{document}
\title[Approximation and Plurisubharmonic Functions]{Plurisubharmonic Approximation and Boundary Values of Plurisubharmonic Functions}
\author[L. Hed]{Lisa Hed}
\address{Department of Mathematics and Mathematical Statistics\\
Ume\aa{} University\\
SE-901 87 Ume\aa{}, Sweden}
\email{Lisa.Hed@math.umu.se}

\author[H. Persson]{H\aa{}kan Persson}
\address{Department of Mathematics\\
Uppsala University\\
SE-751 06 Uppsala, Sweden}
\email{hakan.persson@math.uu.se}

\date{}

\begin{abstract}
	We study the problem of approximating plurisubharmonic functions on a bounded domain $\Omega$ by continuous plurisubharmonic functions defined on neighborhoods of $\bar\Omega$. It turns out that this problem can be linked to the problem of solving a Dirichlet type problem for functions plurisubharmonic on the compact set $\bar\Omega$ in the sense of Poletsky. A stronger notion of hyperconvexity is introduced to fully utilize this connection, and we show that for this class of domains the duality between the two problems is perfect. In this setting, we give a characterization of plurisubharmonic boundary values, and prove some theorems regarding the approximation of  plurisubharmonic functions.
\end{abstract}
\subjclass[2010]{Primary 32U05, 31C10; Secondary 46A55}

\keywords{plurisubharmonic functions on compacts, Jensen measures, monotone convergence}

\hyphenation{plu-ri-sub-harm-on-ic}
\maketitle
\section{Introduction}
The motivation for the current paper is twofold. Firstly, we are interested in the problem of approximating an upper semi-continuous function $u$ defined on the closure of a bounded domain $\Omega \subset \Cn$ and plurisubharmonic on its interior by smooth plurisubharmonic functions defined on neighborhoods of $\bar \Omega$. This type of approximation can be seen as a compound of the approximation problem studied by Forna\ae{}ss and Wiegerinck in \cite{FW} where the attention was restricted to the approximation of \emph{continuous} plurisubharmonic functions and a approximation problem studied by Wikstr\"{o}m \cite{W} and G\"og\"us \cite{Go} (see also \cite{NW}), where the approximants are only suppose to be defined on $\bar\Omega$. Similar approximation problems has also been studied by Q. D. Nguyen, see \cite{Dieu}. G\"{o}g\"{u}s \cite[Example 7.3]{Go} has shown that already the Wikstr\"om type approximation may fail in very nice domains, such as star-shaped, strongly hyperconvex smooth domains. Therefore, it makes sense not only to ask for what domains such an approximation always is possible, but to ask for a characterization of those plurisubharmonic functions $u$ that can be monotonely approximated from outside.  According to a result by the authors and R. Czy\.{z} \cite{CHP}, this is possible if and only if the function $u$ is plurisubharmonic on the compact set $\bar\Omega$.

The notion of a plurisubharmonic function on a compact set in $\Cn$ can be traced back to the works of Rickarts \cite{Ri}, Gamelin \cite{Ga} and Gamelin and Sibony \cite{GS}, but its modern form is definitely due to Poletsky \cite{PAG}. It has successfully been used to answer questions linked to polynomial hulls, polynomial convexity and analytic structure, but in this paper we apply the notion to the field of plurisubharmonic approximation. Postponing the complete definition of plurisubharmonic functions on compact sets to the next section, we content ourself with saying that given a compact set $X \subset \Cn$, one attaches to each point $z\in X$, a class of probability measures $\J{z}{X}$, the so called Jensen measures, and say that an upper semicontinuous function $u$ is plurisubharmonic on $X$ if $u(z) \leq \int u\, d\mu$ whenever $\mu \in \J{z}{X}$.

The main question in this paper is: Given a bounded domain $\Omega$, how can you tell if a function $u \in \PSH{\Omega}$, that is upper semicontinuous on $\bar\Omega$, also is plurisubharmonic in the sense of Poletsky on $\bar \Omega$. One obvious necessary condition to be fulfilled is that the following Dirichlet problem can be solved:
\[
	\begin{cases}
		v \mbox{ is plurisubharmonic on } \bar\Omega, \\
		v=u \mbox{, on } \partial \Omega.
		\end{cases}
\]
One readily establishes that in case $\Omega$ is hyperconvex, this problem can be solved if the boundary function $u \in \PSH{\partial \Omega}$.
 It is however not true that this necessary condition is sufficient. We  therefore ask the question, under which circumstances the plurisubharmonicity in the sence of Poletsky of a function $u \in \PSH{\Omega}$ is solely determined by its boundary values. This leads us to defining a new notion of hyperconvexity, more fitting to our setting. We say that a bounded domain $\Omega \subset \Cn$ is \emph{P-hyperconvex} if there exists a non-constant function $u \in \PSH{\bar\Omega}$ such that $u(z)=0$ for all $z \in \partial \Omega$.

By simple examples, we show that this regularity condition is strictly stronger than hyperconvexity and strictly weaker than the notions of strict hyperconvexity as studied by for example Bremermann \cite{Br}, Nivoche \cite{N} and Poletsky \cite{PsHx}.

To better understand this notion, we give several different characterisations of P-hyperconvexity, all of which has its counterpart for hyperconvex domains.

\begin{mainthm}\label{thm:a}
Let $\Omega \Subset \C^n$ be a domain. The following are equivalent:
\begin{enumerate}
\item $\Omega$ is P-hyperconvex,
\item for every $z \in \partial \Omega$ and every $\mu \in \J{z}{\bar \Omega}$,  $\supp(\mu) \subset \partial \Omega$,
\item for every $z \in \partial \Omega$ there exists a function $\varphi \in \PSH{\bar \Omega}$, $\varphi \not\equiv 0$ such that $\varphi \leq 0$ and $\varphi(z)=0$,
\item for every $z \in \bar\Omega$, there exists a neighborhood $V_z$ of $z$ such that $V_z\cap\Omega$ is P-hyperconvex.
\end{enumerate}
\end{mainthm}
It turns out that this is exactly the condition we have been looking for, and we can show the following theorems.
 \begin{mainthm}\label{thm:c}
 Suppose that  $\Omega$ is a P-hyperconvex domain and that $u$ is plurisubharmonic on $\Omega$ and upper semicontinuous on $\bar\Omega$. If there is a function $v \in \PSH{\bar\Omega}$ such that $u|_{\partial \Omega}=v|_{\partial \Omega}$, then $u\in \PSH{\bar\Omega}$.
 \end{mainthm}
 \begin{mainthm}\label{thm:mainapprox}
Suppose that $\Omega$ is P-hyperconvex and that $u \in \PSH{\Omega}$. For every relatively compact set $E \Subset \Omega$, there is a sequence $\{u_j\}$ of functions smooth and plurisubharmonic in neighborhoods of $\bar\Omega$ such that $u_j(z) \searrow u(z)$ for every $z \in E$. Moreover, if $u$ is bounded from above, one can chose $E:=\Omega$.
\end{mainthm}
\begin{mainthm}\label{thm:b}
	Suppose that $\Omega$ is a P-hyperconvex domain. Then the Dirichlet problem
	\[
	\begin{cases}
		v \mbox{ is plurisubharmonic on } \bar\Omega; \\
		v=f \mbox{, on } \partial \Omega.
		\end{cases}
\]
can be solved if and only if $f \in \PSH{\partial \Omega}$.

Moreover, if $f \in C(\partial \Omega)$, then $v$ can be chosen to belong to $C(\bar\Omega)$.
\end{mainthm}
Here a few words about the different theorems is in order. Theorem \ref{thm:c} says that on P-hyperconvex domains, the question of whether a plurisubharmonic function on domain extends to plurisubharmonic function on the closure of the domain, is determined by its boundary values. This property characterizes P-hyperconvexity among hyperconvex domains, and is one of the reasons why P-hyperconvex domains are natural to work with. Theorem \ref{thm:mainapprox} is a direct consequence of this fact and says that if one is willing to sacrifice a neighborhood of the boundary, all plurisubharmonic functions can be approximated from the outside by continuous plurisubharmonic functions.

Finally Theorem \ref{thm:b} can be interpreted in two interesting ways. On one hand, it characterises the solutions to a Dirichlet type problem for plurisubharmonic functions on compact sets, and should be compared to the Dirichlet problems studied by Poletsky and Sigurdsson \cite{PS}. On the other hand, one should note that in case one puts some mild regularity assuptions on $\bar \Omega$ (see \cite{H} and \cite{AHP}), every continuous function on $\bar\Omega$ that is plurisubharmonic function on $\Omega$ can be uniformly approximated by continuous plurisubharmonic functions defined on neighborhoods of $\bar\Omega$. In this case this theorem fully characterises the boundary values of continous, plurisubharmonic functions. This question is intimetely connected to the inhomogenous Dirichlet problem for the complex Monge-Amp\`{e}re equation, and has  been studied by  for example \AA{}hag, Czy\.z, Lodin and Wikstr\"om \cite{A}, B\l{}ocki \cite{Bl}, Sadullaev \cite{Sa} and Wikstr\"om \cite{W}.
Combining Theorem \ref{thm:c} and Theorem \ref{thm:b}, gives a natural characterization of the approxible plurisubharmonic functions on P-hyperconvex domains. 
\begin{mainthm}\label{thm:e}
	Let $\Omega$ be a P-hyperconvex domain and suppose that $u \in \PSH{\Omega}$ is upper semicontinuous on $\bar\Omega$. Then $u \in \PSH{\bar\Omega}$ if and only if $u|_{\partial \Omega}\in \PSH{\partial\Omega}$.
\end{mainthm}
\section{Notation}
Throughout this paper,  $\Omega$ denotes a bounded domain, that is a bounded, connected and open set, in $\Cn$ and $X$ denotes a compact set in $\Cn$. Given a set $E$, we denote by $E^\circ$ its interior and by $\bar E$ its closure.

As usual we will denote the set of continuous functions on $\Omega$ by $C(\Omega)$ and the set of plurisubharmonic functions on $\Omega$ by  $\PSH{\Omega}$. Later, the set of plurisubharmonic functions on the compact set $X$ will also be denoted by $\PSH{X}$, and the meaning of $\PSH{E}$ will depend on the context. 
The set of upper semicontinuous on $X$ will be denoted by $\USC{X}$. With a measure, we always mean a positive, regular Borel measure  and given a such a measure, $\mu$, we denote by $\supp(\mu)$, the support of $\mu$. The Dirac measure at the point $z$ will be denoted by $\delta_z$. 

\section{Plurisubharmonicity on compact sets} \label{sec:ppsh}
In this section we define the notion of plurisubharmonicity on compact sets. We follow the work of Sibony \cite{Si}, but there is an equivalent characterization based on analytic disks, due to Poletsky (see \cite{PAG} and \cite{CHP}). We also study some properties of this class of functions and compare them with plurisubharmonic functions on open sets. For a good introduction to the properties of plurisubharmonic functions on open sets, see the monograph by Klimek \cite{K}. 

\begin{defin}
Let $u$ be a function on $X$. We say that $u \in \PSHo{X}$ if there exists $U \supset X$ open and a function $\tilde u \in \PSH{U} \cap C(U)$ such that $\tilde u|_X =u$.
\end{defin}

\begin{defin}\label{def:jens1}
	Let $\J{z}{X}$ be the set of all probability measures $\mu$ on $X$ such that
	\[
		u(z) \leq \int u\,d\mu, \quad \forall \ u \in \PSHo{X}.
	\]
\end{defin}

The set $\J z X$ is called the set of \emph{Jensen measures} (for $z$) and this set is always nonempty since the Dirac measure $\delta_z \in \J{z}{X}$. We are now ready to define what we mean with a plurisubharmonic function on a compact set.

\begin{defin}\label{def:psh}
	Let $u$ be an upper semicontinuous function on $X$. We say that $u$ is \emph{plurisubharmonic on $X$} if, for all $z \in X$, it holds that
	\begin{equation}\label{eq:pshineq}
		u(z) \leq \int u \,d\mu, \quad \forall \ \mu \in  \J{z}{X}.
	\end{equation}
	
	We denote the set of all plurisubharmonic functions on $X$ by $\PSH{X}$.
\end{defin}

\begin{rem}
By the definition, we see that $\PSHo{X} \subset \PSH{X}$, and we will later see that this inclusion is proper (see Example \ref{ex:psh_ext}).
\end{rem}

\begin{exa} \label{ex:pshinterior}
If $X$ is a compact set with non-empty interior, $X^\circ$, and $u \in \PSH{X^\circ}$, then $u \in \PSH{X^\circ}$.  This follows from the observation that the push-forward of the normalized arc-length measure by a complex line defines a Jensen measure.
\end{exa}

One of the main reasons for this definition is a duality theorem of D.A. Edwards which we now will describe. Suppose in the following that $X$ is a compact metric space and that $\mathcal F(X)$ is a cone of upper semicontinuous functions on $X$ that contains the constants. For every $z \in X$, we now define $\mathcal{M}^{\mathcal F}_z(X)$ as the set of all probability measures $\mu$ on $X$ such that
\[
	\phi(z) \leq \int \phi\, d\mu, \quad \forall \phi \in \mathcal{F}.
\]
\begin{thm}[Edwards' theorem, \cite{E}]
	Suppose that $X$, $\F(X)$ and $\M{z}{\mathcal F}{X}$ are as above, and that $\phi$ is a lower semicontinuous function on $X$. Then
	\[
	\sup\Bigl\{v(z): v\in \F(X), v \leq\phi\Bigr\} = \inf\left\{\int \phi\,d\mu: \mu \in \M{z}{\mathcal F}{X}\right\}.
	\]
\end{thm}
For a proof of the theorem in this setting, see \cite{W}.

In our work there are two standard choices for $\mathcal F$, $\PSHo{X}$ or $\PSH{X}$. In both cases it follows from the definition that $\M{z}{\mathcal F}{X}=\J{z}{X}$. Applying Edwards' theorem we get the following theorem.
\begin{thm}\label{thm:Edwards}
Let $\phi$ be a lower semicontinuous function on $X$. Then
	\begin{align*}
		\inf\left\{\int \phi \,d\mu: \mu \in \J{z}{X}\right\}
		&=\sup \biggl\{\psi(z): \psi \in \PSHo{X}, \psi \leq\phi \biggr\}\\
		&= \sup \biggl\{\psi(z): \psi \in \PSH{X}, \psi \leq\phi \biggr\}.
	\end{align*}
\end{thm}
With this powerful tool, one can prove the following theorem, which is of outmost importance for our work. 
\begin{thm}[\cite{CHP}] \label{thm:approx}
		Let $X \subset \C^n$ be compact and suppose that $u$ is an upper semicontinuous function on $X$. Then $u \in \PSH{X}$ if and only if there exist functions $u_j \in \PSHo{X}$ such that $u_j \searrow u$.
\end{thm}
\begin{rem}
	Since every function $v \in \PSHo{X}$ can be approximated by a decreasing sequence of \emph{smooth} plurisubharmonic functions defined in neighborhoods of $X$, it follows by a diagonalization argument that the functions $u_j$ in the theorem in fact can assumed to be smooth.
\end{rem}
The functions in $\PSH{X}$ share a lot of nice properties with ordinary plurisubharmonic functions. Here follows some examples of that.

\begin{exa}
If $u,v \in \PSH{X}$ and $s,t \geq 0$, it follows from the properties of the integral that
\begin{enumerate}
\item
	$su+tv \in \PSH{X}$,
\item
	$\max\{u,v\} \in \PSH{X}$.
\end{enumerate}
\end{exa}
\begin{exa}\label{ex:monotone_seq}
	Suppose that  $u_j \in \PSH{X}$ and $u_j \geq u_{j+1}$. Then it follows from the monotone convergence theorem that $\lim u_j \in \PSH X$.
\end{exa}
\begin{exa}
	The function $u \in C(X)$ is plurisubharmonic on $X$ if and only if there are functions $u_j \in \PSHo X$ such that $u_j \to u$ uniformly on $X$. This follows immediately from Theorem \ref{thm:approx} and Dini's theorem.
\end{exa}
\begin{exa}
	If $u$ is plurisubharmonic in some neighborhood of $X$, then $u|_X \in \PSH{X}$. This follows from the basic fact that $u|_X$ may be approximated by a decreasing sequence $u_j \in \PSHo X$.
\end{exa}
This last example is of principal interest, since it guarantees that for any $X$, there exist plenty of plurisubharmonic functions on $X$. On the other hand, it is interesting to know if there are other plurisubharmonic functions of $X$ besides restrictions of ordinary plurisubharmonic functions. The following example gives an answer to this question..
\begin{exa}\label{ex:psh_ext}
	Let $\D$ denote the unit disk in $\C$ and consider the plurisubharmonic function 
	\[
	u(z)=-\sqrt{1-\abs{z}^2}
	\]
	defined on $X:=\bar\D$. Since it can be approximated by the bounded sequence $u_j(z)=u\bigl((1-1/j)z\bigr)$, it belongs to $\PSH{\bar \D}$. Since $\Delta u$ tends to infinity as $z$ tends to $\partial X$, it cannot be extended to a subharmonic function in any neighborhood of $X$. Using the theory of the complex Monge--Amp\`{e}re operator, this example can also be extended to higher dimensions. See \cite{C} for the details.
\end{exa}

It is not at all clear from the definition of $\PSH X$, whether it is a local property to be plurisubharmonic on a compact set. However, Gauthier has shown that it is a local property to be uniformly approximable by functions in $\PSHo X$. Using Theorem \ref{thm:approx}, we can therefore deduce the following localization theorem. 

\begin{thm}[\cite{Gau}] \label{thm:Loc}
A function $u \in \PSH{X}\cap C(X)$ if, and only if, for each $z \in X$, there is a neighborhood $B=B_z$ such that $u|_{X \cap \bar B} \in \PSH{X \cap \bar B}\cap C(X\cap \bar B)$.
\end{thm}

We now turn to the question of when a function $u \in \PSH{\Omega}\cap\USC{\bar\Omega}$ belongs to $\PSH{\bar\Omega}$. By definition, one determines this by integrating against all Jensen measures for every $z \in \bar \Omega$. The next theorem shows that it in fact suffices to do this for $z \in \partial \Omega$.

\begin{thm} \label{thm:bvalues}
Let $\Omega \Subset \C^n$ be a domain. If $\varphi \in\PSH{\Omega}\cap\USC{\bar\Omega}$ and for every $z \in \partial \Omega$ we have that
\begin{equation}\label{eq:ineq}
\varphi(z) \leq \int \varphi \, d\mu \quad \forall \ \mu \in \J{z}{\bar \Omega},
\end{equation}
then $\varphi \in \PSH{\bar\Omega}$.
\end{thm}

\begin{proof}
Since $\phi$ is upper semicontinuous, it can be approximated by a strictly descending sequence $\{\phi_j\} \in \PSHo{\bar\Omega}$. We claim that we can find functions $\{v_j\} \in \PSHo{\bar\Omega}$ such that $v_j \leq \phi_j$ and $v_j(z) \searrow \phi(z)$ for every $z \in \partial \Omega$. From this it follows that the functions
\[
	u_j(z)=\begin{cases} \max\bigl\{\phi(z),v_j(z)\bigr\}, & \mbox{if } z \in \bar\Omega;\\
					v_j(z),& \mbox{else.}
		\end{cases}
\]
will all belong to $\PSHo{\bar\Omega}$, and $u_j\to \phi$. It now follows from the dominated convergence theorem that $\phi \in \PSH{\bar\Omega}$.

It remains to prove the claim. This will be done using the same idea as in \cite{CHP}. Let
\[
F_j(z):=\sup\{\varphi(z): \varphi \in \PSHo{\bar\Omega}, \varphi \leq \phi_j \}.
\]
It then follows from Edwards' theorem (Theorem \ref{thm:Edwards}) that 
\[
F_j(z)=\inf\{\int \phi_j \, d\mu: \mu \in \J{z}{\bar\Omega}\}.
\]
Since $\J{z}{\bar\Omega}$ is compact in the weak-* topology, we can for each $z \in \bar\Omega$ find $\mu_z \in \J{z}{\bar\Omega}$ such that $F_j(z)=\int \phi_j \, d\mu_z.$ It then holds that $\phi_j \geq F_j$, and 
\[
F_j(z)=\int \phi_j \, d\mu_z > \int \phi \, d\mu_z \geq \phi(z),\quad \forall z \in \partial \Omega.
\]
By the construction of $F_j$ we know that for every given $z \in \partial\Omega$, there exists a function $v_z \in \PSHo{\bar\Omega}$ such that $v_z \leq F_j$ and $\phi(z)<v_z(z) \leq F_j(z)$. The function $\phi-v_z$ is upper semi-continuous and therefore the set $U_z:=\{w \in \partial\Omega: \phi(w)-v_z(w)<0\}$ is open in $\partial\Omega$. It now follows from the compactness of $\partial\Omega$ that there are finitely many points $z_1,...,z_k$ with corresponding functions $v_{z_1},...,v_{z_k}$ and open sets $U_{z_1},...,U_{z_k}$ such that $u < v_{z_j}$ in $U_{z_j}$ and $\partial\Omega=\bigcup_{j=1}^k U_{z_j}$. The function $v_j=\max\{v_{z_1},...,v_{z_k}\}$ belongs to $\PSHo{\bar\Omega}$ and $\phi(z)< v_j(z) \leq \phi_j(z)$ for $z\in \partial \Omega$. This completes the proof.
\end{proof}
In some sense, this theorem  seems to suggest that the question of plurisubharmonicity on $\bar\Omega$ can be localized to the boundary. This is however not the case, since for a point $z \in \partial\Omega$ a measure $\mu \in\J{z}{\bar\Omega}$, could have support off $\partial \Omega$, thus collecting information from the interior of $\Omega$. This leads us into the next section and the notion of P-hyperconvexity.

\section{P-hyperconvexity} \label{sec:phx}
We will now introduce a new notion of regularity of a domain called P-hyperconvexity. This notion is quite natural when working with approximation of plurisubharmonic functions, and positions itself between other commonly used notions of regularity. We remind the reader of two of those.

\begin{defin}
	Suppose that $\Omega$ is a domain in $\Cn$. If there exists a non-constant  function $\phi \in \PSH \Omega \cap C(\bar \Omega)$ such that $\phi(z) = 0$ for all $z\in \partial \Omega$, we say that $\Omega$ is \emph{hyperconvex} and that $\phi$ is a \emph{negative plurisubharmonic exhaustion function}.
\end{defin}
If $\Omega$ is a hyperconvex domain with negative plurisubharmonic exhaustion function $\psi$, the function $-1/\psi$ is a plurisubharmonic exhaustion function for $\Omega$. Hence pseudoconvexity is a necessary condition for hyperconvexity. The punctured disk in $\C$ shows that this condition is not sufficient, but Demailly \cite{D} has shown that every pseudoconvex domain whose domain locally can be written as a the graph of a Lipschitz function, is hyperconvex.

For future reference, we collect some well-known properties of hyperconvex domains.
\begin{thm}[\cite{KR}]\label{thm:hx_lokal}
	The domain $\Omega$ is hyperconvex if and only if for all $z \in \partial \Omega$, there exists a neighborhood $U$ of $z$ such that $\Omega\cap U$ is hyperconvex.
\end{thm}
\begin{thm}[\cite{CCW}]\label{thm:jensen_supp}
	The domain $\Omega$ is hyperconvex if and only if the following holds: If $z \in \partial \Omega$ and $\mu$ is a probability measure on $\bar\Omega$ such that 
	\[
		u(z) \leq \int u \,d\mu, \quad \forall u \in \PSH\Omega,
	\]
	then it follows that $\supp(\mu) \subset \partial \Omega$.
\end{thm}

Another notion of regularity is that of strict hyperconvexity. For examples of the uses of this notion, see for example the works of Bremermann \cite{Br}, Nivoche \cite{N}, or Poletsky \cite{PsHx}.
\begin{defin}
	Suppose that $\Omega$ is a domain in $\Cn$. If there exists a neighborhood $U$ of $\bar\Omega$ and a function $\phi \in \PSH U \cap C(U)$ such that 
	\[
	\{z\in U: \phi(z) < 0\}=\Omega,
	\]
	we say that $\Omega$ is \emph{strictly hyperconvex}.
\end{defin}
\begin{rem}
	There seems to be no consensus on the definition of strict hyperconvexity, and often even stricter definitions are used, see for example \cite[p. 417]{N}. This rather general definition is however sufficient for our needs.
\end{rem}
A strictly hyperconvex domain is thus a hyperconvex domain with a negative plurisubharmonic exhaustion function that can be plurisubharmonically extended to some neighborhood of the closure of the domain. Classical example of domains of this type are bounded strictly pseudoconvex domains with $C^2$-boundary and analytic polyhedra. The following example shows that strict hyperconvexity puts a rather strong global restriction on the complex structure of the boundary.
\begin{exa}\label{ex:worm}
	The so-called worm domain, of Diederich and Forn\ae{}ss \cite{DF} is a bounded pseudoconvex domain with $C^\infty$-boundary  in $\C^2$ that has several interesting properties. Here we give a brief description of one construction of the worm domain. For a more thorough presentation, see \cite{KP}.
	
	Let $W$ be the domain defined by
	\[
		W:=\{(z,w) \in \C^2: \abs{z-e^{i\log{\abs w^2}}}^2<1-\eta(\log{\abs w^2})\},
	\]
		where $\eta$ is a function satisfying.
		\begin{enumerate}
		\item
			$\eta \geq 0$, $\eta$ is even and convex;
		\item
			$\eta^{-1}(0) = [-2\pi,2\pi]$;
		\item
			$\eta(x) > 1$ if $\abs{x} > a$, for some $a > 0$.
		\item
			$\eta'(x) \neq 0$ if $\eta(x) = 1$.
		\end{enumerate}
		It can be shown that such an $\eta$ exists, and direct calculations show that $W$ is both smooth and pseudoconvex. It can also be shown that for every $z \in \partial W$ there is a neighborhood $N$ of $z$ such that $N\cap W$ is strictly hyperconvex. Nevertheless, it can easily be seen that $W$ is not strictly hyperconvex. Indeed, suppose that this was the case, then there would have existed a neighborhood $U$ of $\bar W$ and a plurisubharmonic function $\psi$ such that $W = \{(z,w) \in U: \psi(z,w)<0\}$. Now since $(0,w)\in \partial W$ for every $w$ such that $\abs{\log{\abs{w}^2}}\leq 2\pi$, all points of the form $(\epsilon, w)$ lie in $U$ if $\epsilon >0$ is small enough, and $\abs{\log{\abs w^2}}\leq 2\pi$. The function $\psi_\epsilon(\zeta):=\psi(\epsilon, \zeta)$ is therefore a subharmonic function on the annulus $e^{-2\pi}<\abs{\zeta}<e^{2\pi}$. Since the point $(\epsilon, w)$ lies in $W$ if $\abs{\log{\abs{w}^2}}=2\pi$, but lies in $U \setminus \bar W$ if $\abs{\log{\abs{w}^2}}=\pi$, $\psi_\epsilon$ will be negative on its boundary, but non-negative (at least) on the interior points where $\abs{\zeta}=e^{\pi/2}$. This contradicts the maximum principle.
\end{exa}
Here we propose a new, intermediate notion of hyperconvexity.

\begin{defin}
A domain $\Omega \Subset \C^n$ is called \emph{P-hyperconvex} if there is a non-constant function $\phi \in \PSH{\bar \Omega}\cap C(\bar \Omega)$ such that $\phi(z)=0$ for all $z \in \partial \Omega$.
\end{defin}
It follows from Example \ref{ex:pshinterior} that every P-hyperconvex domain is hyperconvex, and we will soon see that this inclusion is strict. For this we need the following easy proposition. 

\begin{prop}\label{prop:fat}
If $\Omega$ is a P-hyperconvex domain, then $\Omega$ is fat, i.e. $\Omega=(\bar \Omega)^{\circ}$.
\end{prop}
\begin{proof}
Assume that $\Omega$ is P-hyperconvex but not fat. Then $\Omega$ has a negative plurisubharmonic exhaustion function $\psi \in \PSH{\bar \Omega}\cap C(\bar \Omega)$, $\psi|_{\partial \Omega}=0$. Assume that $z \in \partial \Omega \cap (\bar \Omega)^\circ$, then $\psi(z)=0$. But this means that $\psi$ attains a maximum in an interior point, since it was seen in Example \ref{ex:pshinterior} that $\psi \in \PSH{\Omega}$, this contradicts the maximum principle.
\end{proof}
\begin{exa}\label{ex:hx_ej_phx}
Let  $U=\mathbb{D}\setminus [-\frac{1}{2},\frac{1}{2}]$, where $\D$ denotes the unit disk in $\C$. Since $U$ is regular for the Laplace equation, it is hyperconvex. However, since $U$ is not fat, it is not P-hyperconvex. By considering $\Omega: = U \times \D \subset \C^2$, we get an example in higher dimension.
\end{exa}
\begin{rem}
 It follows from \cite[Theorem 6.10]{Gar} that all fat hyperconvex domains in $\C$ are P-hyperconvex. It would be very interesting to know whether something like this holds also in higher dimensions. It follows from \cite{AHP} that a counterexample would have to have a very irregular boundary. 
\end{rem}

It follows from the definition that every strictly hyperconvex domain is P-hyperconvex. The following example shows that the inclusion is strict. In fact it seems like P-hyperconvexity is much weaker than strict hyperconvexity.
\begin{exa}
Let $\Omega$ be a bounded pseudoconvex domain with $C^1$-boundary. By Kerzman and Rosay \cite{KR},  any such domain is hyperconvex, and thus $\Omega$ has a negative plurisubharmonic exhaustion function $\psi \in \PSH\Omega\cap C(\bar\Omega)$. By Forn\ae{}ss and Wiegerinck  \cite{FW}, $\PSH U \cap C(\bar U)=\PSH{\bar{U}} \cap C(\bar U)$, for any bounded domain $U$ with $C^1$-boundary. Putting these results together, we see that $\Omega$ is P-hyperconvex. This means in particular that the worm domain of Diederich and Forn\ae{}ss (see Example \ref{ex:worm}) is P-hyperconvex.
\end{exa}

The following proposition gives a geometric description of the boundary of a P-hyperconvex domain.
\begin{prop}\label{thm:analytic_d_boundary}
	Let $\Omega$ be a P-hyperconvex domain and let $f$ be a continuous mapping from $\bar\D	$ to $\bar\Omega$ that is holomorphic on $\D$. If there is a point $\zeta_0 \in \D$ such that $f(\zeta_0) \in \partial \Omega$, then $f(\bar \D) \subset \partial \Omega$.
\end{prop}

\begin{proof}
	Since $\Omega$ is P-hyperconvex, there is a function $\psi \in \PSH{\bar\Omega}\cap C(\bar\Omega)$ such that $\psi(z)<0$ in $\Omega$ and $\psi(z)=0$ for all $z \in \partial \Omega$. Due to Theorem \ref{thm:approx}, there are $\psi_j \in \PSHo{\bar\Omega}$ such that $\psi_j \searrow \psi$ on $\bar\Omega$. We now let $u_j:=\psi_j\circ f$ and $u:=\psi\circ f$. Then $u_j$ is subharmonic on $\D$ and since $u_j \searrow u$ on $\D$, $u$ is also subharmonic on $\D$. Since $\sup u\leq\sup\psi = 0$ and $u(\zeta)=0$ it follows from the maximum principle that $u \equiv 0$. This means that $f(\bar\D)\subset \partial \Omega$.
\end{proof}
	An example by Carlehed, Cegrell and Wikstr\"om \cite[Example 2.11]{CCW}, shows that Proposition \ref{thm:analytic_d_boundary} can not be used to fully characterize P-hyperconvex domains. The following example shows however, that in some cases, the proposition can be used to distinguish between a P-hyperconvex domain, and a domain that is merely pseudoconvex. 
\begin{exa}
	Let $\Omega$ be the so-called Hartogs triangle defined by
	\[
		\Omega:=\{(z,w)\in \C^2: \abs{z}<\abs{w} <1\}.
	\]
	Since $\phi: (z,w)\mapsto(\frac{z}{w},w)$ is a biholomorphic mapping from $\Omega$ to the pseudoconvex domain $\{(z,w): \abs{z} < 1, 0 < \abs{w} < 1\}$, the domain $\Omega$ is pseudoconvex. However, the mapping $f: \D \to \bar\Omega$ given by $\zeta \mapsto (0,\zeta)$ is an anlytic disk with $f(0)\in \partial \Omega$, but $f(\bar\D)\not\subset \partial \Omega$. By Proposition \ref{thm:analytic_d_boundary}, $\Omega$ is not P-hyperconvex.
\end{exa}
\begin{rem}
	After some extra thought, one realizes that the Hartogs triangle is not even hyperconvex.
\end{rem}
\begin{prop} \label{prop:snitt}
Let $\Omega_1$ and $\Omega_2$ be P-hyperconvex domains in $\C^n$, then $\Omega:=\Omega_1 \cap \Omega_2$ is P-hyperconvex.
\end{prop}

\begin{proof}
Let $\psi_1 \in \PSH{\bar {\Omega}_1}$ and $\psi_2 \in \PSH{\bar {\Omega}_2}$ be the exhaustion functions for $\Omega_1$ respectively $\Omega_2$. Then the function
\[
	\psi(z):=\max\{\psi_1(z), \psi_2(z)\},
\]
is an negative plurisubharmonic exhaustion function for $\Omega$ in $\PSH{\bar\Omega}$.
\end{proof}

\begin{rem}
By a similar argument as in Proposition \ref{prop:snitt} we can observe that if $\Omega_1$ and $\Omega_2$ are P-hyperconvex domains, so is $\Omega=\Omega_1 \times \Omega_2$.
\end{rem}

In the same way as hyperconvexity can be fully characterized by the support of a class of measures (see Theorem \ref{thm:jensen_supp}), the notion of P-hyperconvexity has a corresponding characterization using the measures $\J{z}{\bar \Omega}$, for $z \in \partial \Omega$.

\begin{thm}[Theorem \ref{thm:a}] \label{thm:phxb}
Let $\Omega \Subset \C^n$ be a domain. The following are equivalent:
\begin{enumerate}
\item $\Omega$ is P-hyperconvex;
\item for every $z \in \partial \Omega$ and every $\mu \in \J{z}{\bar \Omega}$, we have that $\supp(\mu) \subset \partial \Omega$;
\item for every $z \in \partial \Omega$ there exists a non-constant function $\varphi \in \PSH{\bar \Omega}$ such that $\varphi \leq 0$ and $\varphi(z)=0$;
\item for every $z \in \partial \Omega$ there exists a neighborhood $V_z$ such that $\Omega \cap V_z$ is P-hyperconvex.
\end{enumerate}
\end{thm}
\begin{proof}
$(1) \Rightarrow (4)$ Follows from Proposition \ref{prop:snitt}.

$(4) \Rightarrow (3)$ It follows from Theorem \ref{thm:hx_lokal} that $\Omega$ is hyperconvex. Hence, there exists $\psi \in \PSH{\Omega}\cap C(\bar \Omega)$, $\psi \not\equiv 0,$ such that $\psi|_{\partial \Omega}=0$. We want to show that $\psi \in \PSH{\bar \Omega}$. By Theorem \ref{thm:Loc} it is enough to show that for every $z_0 \in \bar \Omega$, there is a ball $B_{z_0}$ such that $\psi \in \PSH{\bar{\Omega \cap B_z}}$. For $z_0 \in \Omega$ this is obviously true, so it is enough to consider $z_0 \in \partial \Omega$. From now on, let $z_0$ be a point in $\partial \Omega$ and $B_{z_0}$ a small ball centered at $z_0$ such that $B_{z_0} \Subset V_{z_j}$, for some $j$. By Theorem \ref{thm:bvalues}, it is enough to show that for $z \in \partial(\Omega\cap B_{z_0})$,
\begin{equation}\label{eq:locally}
	\psi(z)\leq \int \psi \, d\mu, \quad \forall \mu \in \J{z}{\bar{\Omega \cap B_{z_0}}}.
\end{equation}
Suppose first that $z \in \partial \Omega \cap B_{z_0}$. Since $\psi \equiv 0$ on $\partial \Omega$, the inequality \eqref{eq:locally} will be shown to hold if we can show that every $\mu \in \J{z}{\bar{\Omega \cap B_{z_0}}}$ only has support on $\partial \Omega$. Since $\Omega \cap V_{z_j}$ is P-hyperconvex, it has an exhaustion function $\varphi \in \PSH{\bar{\Omega \cap {V}_{z_j}}}$. Suppose  $\mu \in \mathcal{J}_z(\bar{\Omega \cap B_{z_0}})$, then
$$0=\varphi(z)\leq \int \varphi \, d\mu \leq 0.$$
Hence, $\mu$ has only support where $\varphi=0$, that is on $\partial \Omega$.\\

Now suppose that $z \in \Omega \cap \partial B_{z_0}$. In this case inequality \eqref{eq:locally} holds, since $\mathcal{J}_{z}(\bar{\Omega \cap B_{z_0}})=\{\delta_z\}$. To see this, we begin by writing $z$ on the form $z=z_0+rv$, for some unit vector $v$ and some real $r>0$ and define the function
\[
	u(\zeta)=\abs{\zeta-z_0+rv}-2r.
\]
For any $\mu \in \mathcal{J}_{z}(\bar{\Omega \cap B_{z_0}})$ it now holds that
\[
	0=u(z) \leq \int u \, d\mu \leq 0,
\]
and since $u$ is strictly negative on $\bar{\Omega \cap B_{z_0}}\setminus \{z\}$, it follows that $\mu = \delta_z$.

$(3) \Rightarrow (2)$ Assume that  $z \in \partial \Omega$ and $\mu \in \J{z}{\bar \Omega}$. By assumption there exists a non-constant function $\varphi \in \PSH{\bar \Omega}$ such that $\varphi \leq 0$ and $\varphi(z)=0$. Then
\[
0=\varphi(z) \leq \int_{\bar \Omega}\varphi \, d\mu \leq 0,
\]
and since $\phi < 0$ in $\Omega$, $\supp(\mu) \subset \partial \Omega$.

$(2) \Rightarrow (1)$ Now assume that for $z \in \partial \Omega$, every measure $\mu \in \J{z}{\bar \Omega}$ is supported on $\partial \Omega$. Then it follows from Theorem \ref{thm:jensen_supp} that $\Omega$ is hyperconvex, and hence $\Omega$ has a negative exhaustion function $\psi \in \PSH{\Omega} \cap C(\bar \Omega)$. We want to show that $\psi \in \PSH{\bar \Omega}$. By Theorem \ref{thm:bvalues} it is enough to show that for every $z \in \partial \Omega$ we have that
\[
\psi(z) \leq \int \psi \, d\mu \quad \forall \ \mu \in \J{z}{\bar \Omega},
\]
and this follows directly from the assumption that all Jensen measures are supported on $\partial \Omega$.
\end{proof}
This theorem, together with Theorem \ref{thm:bvalues}, shows that on P-hyperconvex domains, the question of whether a function $u \in \PSH{\Omega}\cap\USC{\bar\Omega}$, belongs to $\PSH{\bar\Omega}$ is really localized to the boundary. This observation is so important that we state it as a theorem.
\begin{thm}[Theorem \ref{thm:c}]\label{thm:psh_boundary_coin}
	Suppose that  $\Omega$ is a P-hyperconvex domain and that $u \in \PSH{\Omega}\cap\USC{\bar\Omega}$. If there is a function $v \in \PSH{\bar\Omega}$ such that $u|_{\partial \Omega}=v|_{\partial \Omega}$, then $u\in \PSH{\bar\Omega}$.
\end{thm}
\begin{rem}
	If $\Omega$ is a hyperconvex domain for which the conclusion of the theorem holds, then $\Omega$ is also P-hyperconvex. Inded, since the constant function $v=0$ is plurisubharmonic on $\bar\Omega$, any negative plurisubharmonic exhaustion function of $\Omega$ has to be plurisubharmonic on $\bar\Omega$. In this sense this theorem characterizes P-hyperconvexity among hyperconvex domains.
\end{rem}
A special case of the previous theorem is when $u$ and $v$ are negative plurisubharmonic exhaustion functions for $\Omega$. The definition of P-hyperconvexity only demands the existence of one negative plurisubharmonic exhaustion function that belongs to $\PSH{\bar\Omega}$, but according to this theorem, it is equivalent to demand that all such exhaustion functions belong to $\PSH{\bar\Omega}$.

\begin{cor}[Theorem \ref{thm:mainapprox}]
Suppose that $\Omega$ is P-hyperconvex and that $u \in \PSH{\Omega}$. For every relatively compact set $E \Subset \Omega$, there is a sequence $u_j \in \PSHo{\bar\Omega}$ such that $u_j(z) \searrow u(z)$ for every $z \in E$. Moreover, if $u$ is bounded from above, one can chose $E:=\Omega$.
\end{cor}
\begin{rem}
	This generalizes Theorem 2 in \cite{FW}.
\end{rem}
\begin{proof}
	For the first statement, let $\psi \in \PSH{\bar\Omega}\cap C(\bar{\Omega})$ be a negative plurisubharmonic exhaustion function for $\Omega$ and let $\epsilon>0$ be so small that $E \subset \{z \in \Omega: \psi(z)< -\epsilon\}=:\Omega_\epsilon$. Then let $M$ be such that $u < M$ on $\Omega_\epsilon$ and let $K >0$ be such that $K(\psi+\epsilon) < u-M$ on $E$. The function $\tilde u$ defined by
\[
	\tilde u(z) :=
	\begin{cases}
		\max\{K(\psi(z)+\epsilon), u-M\},& \mbox{if $z \in \Omega_\epsilon$}, \\
		K(\psi(z)+\epsilon),& \mbox{if $z \in \Omega\setminus \Omega_\epsilon$},
	\end{cases}
\]
will be plurisubharmonic in $\Omega$, and since it has constant boundary values and $\Omega$ is P-hyperconvex, it follows from Corollary \ref{thm:psh_boundary_coin} that $\tilde u \in \PSH{\bar\Omega}$. This proves the first statement of the theorem.
	
	Let us now assume that $u$ is bounded from above. It then makes sense to define  $M:=\sup u$ and let
	\[
		\tilde u(z):=
		\begin{cases}
			u(z), & \mbox{if $z \in \Omega$}, \\
			M,	& \mbox{if $z \in \partial \Omega$}.
		\end{cases}
	\]
	Then $\tilde u \in \PSH{\Omega}\cap\USC{\bar\Omega}$, and since $\tilde u|_{\partial \Omega}\in \PSH{\partial \Omega}$, it follows from Theorem \ref{thm:psh_boundary_coin} that $\tilde u \in \PSH{\bar\Omega}$.
\end{proof}

\begin{lem}\label{lem:psho_pshutv}
	Suppose that $\Omega$ is P-hyperconvex. If $f \in \PSH{\partial\Omega}\cap C(\partial\Omega)$, then there is a function $F \in \PSH{\bar\Omega}\cap C(\bar\Omega)$ such that $F|_{\partial\Omega}=f|_{\partial\Omega}$.
\end{lem}
\begin{proof}
Since $f$ can be monotonely approximated by smooth plurisubharmonic functions defined in neighborhoods of $\partial \Omega$, it is enough to prove the theorem for smooth $f$. In order to do so, let $\psi \in \PSH{\bar \Omega}\cap C(\bar \Omega)$ be a negative plurisubharmonic exhaustion function for $\Omega$. By Theorem \ref{thm:approx} there is an increasing sequence of  functions $\psi_j \in \PSH{\Omega_j}\cap C(\Omega_j)$, where $\bar \Omega \subset \Omega_j$, such that $\psi_j \rightarrow \psi$ uniformly on $\bar \Omega$. Since $\Omega$ is hyperconvex, it follows from a theorem by Kerzman and Rosay \cite[Proposition 1.2]{KR} that $\Omega$ has a negative \emph{strictly} plurisubharmonic exhaustion function $\phi$.

Let $U$ be an open set such that $\partial \Omega \subset U$ and $f \in \PSH{U}\cap C^{\infty}(U)$, and let $V$ be an open set such that $\partial \Omega \subset V \Subset U$. Finally, let $K$ be a compact set such that $\bar\Omega \subset K \cup U$ and $\partial K \subset V$.
	
	We now choose $M>1$  large enough so that
	\[
		\phi(z) - 1 > M \psi(z)	\quad  \forall \ z \in K.
	\]
	Since $\psi_j \rightarrow \psi$ uniformly on $\bar \Omega$, we may assume that $\psi(z)-\psi_j(z) < \frac{1}{Mj}$, for all $z \in \bar\Omega$. We now let
	\[
		\tilde \psi_j :=
		\begin{cases}
			\max\{\phi-\frac{1}{j}, M \psi_j\}, & z \in \Omega,\\
			M\psi_j, & z \in \Omega_j \setminus \Omega.
		\end{cases}
	\]
	The function $\tilde\psi_j$ is plurisubharmonic and continuous in $\Omega_j$ and $\tilde \psi_j= \phi-\frac{1}{j}$ on $K$.
	
	Now let $\theta$ be a smooth function such that $\theta(z)= 1$ for $z \in V$ and $\supp(\theta)\subset U$. Since $\tilde\psi_j$ is strictly plurisubharmonic where $\theta$ is non-constant, we can choose $C$ so large that the function
	\[
		F_j:=C\tilde \psi_j+\theta f,
	\]
	belongs to $\PSHo{\bar\Omega}$. We now note that 
	\[
		\max\{\phi, M \psi\}-\max\{\phi-\frac{1}{j}, M \psi_j\} \leq \frac{1}{j},
	\]
	and let $\tilde\psi=\max\{\phi, M \psi\}$ and $F:=C\tilde \psi+ \theta f$. We then have
	\[
		F \geq  F_j \geq F-\frac{1}{j},
	\]
	so it follows by uniform convergence that $F\in \PSH{\bar\Omega}\cap C(\bar\Omega)$. Furthermore, since for $z \in \partial \Omega$,
	\[
		0 \leq f(z)-F_j(z)=-C\tilde \psi_j(z)=-CM\psi_j \leq \frac{C}{j},
	\]
	we see that $F(z)=f(z)$ on $\partial \Omega$.
\end{proof}
Since P-hyperconvexity is characterized by the property that all Jensen measures for the boundary, has support on the boundary, the question arises, what difference there is between $\J{z}{\bar\Omega}$ and $\J{z}{\partial\Omega}$, for $z \in \partial \Omega$. The following theorem shows that for P-hyperconvex $\Omega$, there is no difference at all, and this property characterizes P-hyperconvexity.
\begin{thm} \label{thm:phxjensen}
	If $\Omega$ is  P-hyperconvex, then $\J{z}{\partial \Omega}=\J{z}{\bar \Omega}$ for every $z \in \partial \Omega$.
\end{thm}

\begin{proof}
	Since $\PSHo{\bar\Omega}|_{\partial \Omega} \subset \PSHo{\partial \Omega}$, $\J{z}{\partial \Omega} \subset \J{z}{\bar \Omega}$ for all $z\in \partial\Omega$. Hence it suffices to show the reverse inclusion. Assume that $f \in \PSHo{\partial \Omega}$. Then, by Lemma \ref{lem:psho_pshutv}, there is a function $F\in \PSH{\bar\Omega}$ such that $F(z)=f(z)$ for $z \in \partial \Omega$. Then for $z \in \partial \Omega$ and $\mu \in \J{z}{\bar \Omega}$,
	\[
		f(z) =F(z) \leq \int F \, d\mu = \int f \, d\mu,
	\]
	where the last integral makes sense, since $\mu$ is only supported on $\partial \Omega$.
\end{proof}

\section{Plurisubharmonic extension} \label{sec:pshext}
In this section we want to study those functions $f \in \USC{\partial \Omega}$ which are boundary values of a function $F \in \PSH{\bar \Omega}$. We will see that if $\Omega$ is P-hyperconvex, a natural characterisation in terms of plurisubharmonicity on the boundary is possible.

We begin with some comments on when $f \in C(\partial \Omega)$ can be extended to a function $F \in \PSH{\Omega}\cap C(\bar \Omega)$.

\begin{prop}\label{prop:hx_ext}
Let $\Omega$ be a bounded hyperconvex domain in $\Cn$ and $f \in \PSH{\partial \Omega}\cap C (\partial\Omega)$. Then there exists a function $F \in \PSH{\Omega}\cap C(\bar\Omega)$ such that $F|_{\partial \Omega}=f$.
\end{prop}
\begin{proof}
It is shown in \cite{HHD} that every $u \in \PSHo{\partial \Omega}$ can can be extended to a function in $\PSH{\Omega}\cap C(\bar\Omega)$. Now \cite[Theorem 3.5]{W} implies the existence of a function $F$ as in the statement of the theorem.
\end{proof}
The above proposition has two flaws: it cannot make any claims of necessity, and it does not give any information whether the extended function $F$ is plurisubharmonic on the compact $\bar\Omega$. To remedy these problems, one has to impose extra regularity on $\Omega$, and in this context, the notion of P-hyperconvexity is very fitting.

\begin{cor}[Theorem \ref{thm:b}]\label{thm:ppsh_ext}
Let $\Omega \Subset \C^n$ be a P-hyperconvex domain and let $f \in \mathcal{USC}(\partial \Omega)$. Then the following are equivalent:
\begin{enumerate}
	\item
		there exists $F \in \PSH{\bar \Omega}$ such that $F|_{\partial \Omega}=f$,
	\item
		$f \in \PSH{\partial \Omega}$.
\end{enumerate}
Furthermore, if $f$ is continuous, the function $F$ can be chosen to belong to $\PSH{\bar\Omega}\cap C(\bar\Omega)$.
\end{cor}
\begin{proof}
$(1) \Rightarrow (2)$ This follows from the definition of $\PSH{\bar\Omega}$ and Theorem \ref{thm:phxjensen}.

$(2) \Rightarrow (1)$ Since $f \in \PSH{\partial\Omega}$, there is a sequence $f_j \in \PSHo{\partial \Omega}$ such that $f_j \searrow f$. Since $\Omega$ is P-hyperconvex, it is regular with respect to the Laplace equation, and hence there is a continuous function $H_j$ on $\bar\Omega$ that is harmonic on $\Omega$, such that $H_j(z)=f_j(z)$ for $z\in \partial \Omega$. Let
\[
\Phi_j(z) := \sup\{u(z): u\in \PSH{\bar\Omega}, u \leq H_j\}.
\]
It follows from Theorem \ref{thm:Edwards} that
\[
\Phi_j(z) = \sup\{u(z): u\in \PSH{\bar\Omega}\cap C(\bar\Omega), u \leq H_j\},
\]
so being the supremum of a family of continuous functions, $\Phi_j$ is lower semicontinuous. We will show that it in fact is continuous. For this, denote by $\Phi_j^*$ the upper semicontinuous regularization of $\Phi_j$. Then $\Phi_j^* \in \PSH{\Omega}\cap\USC{\bar\Omega}$.  It follows from Lemma \ref{lem:psho_pshutv} that there is a function $F_j \in \PSH{\bar\Omega}\cap C(\bar\Omega)$ such $F_j(z)=f_j(z)$ for $z \in \partial \Omega$, and therefore we have that
\[
	\Phi_j(z) \geq F_j(z)=f_j(z), \quad \forall z \in \partial\Omega.
\]
On the other hand we also have that 
\[
	\limsup_{\Omega \ni w \to z} \Phi_j(w) \leq \limsup_{\Omega \ni w \to z} H_j(w)=H_j(z),
\]
so we see that $\Phi_j^*\in \PSH{\Omega}\cap\USC{\bar\Omega}$, and shares boundary values with the function $F_j \in \PSH{\bar\Omega}$. It follows from Theorem \ref{thm:psh_boundary_coin} that $\Phi_j^*\in \PSH{\bar\Omega}$, and therefore $\Phi_j^*$ belongs to the constituting family for $\Phi_j$. This means that $\Phi_j=\Phi_j^*$, and hence $\Phi_j\in \PSH{\bar\Omega}\cap C(\bar\Omega)$. Now let
\[
\Phi = \lim_{j\to \infty} \Phi_j.
\]
By the construction of $\Phi_j$, it follows that $\Phi|_{\partial \Omega}=f$, and being a decreasing limit of functions in $\PSH{\bar\Omega}$,
 $\Phi$ also belsongs to $\PSH{\bar\Omega}$.

Now suppose that $f \in C(\partial \Omega)$, and let $H$ be the continuous function on $\bar\Omega$ that is harmonic on $\Omega$. Similarly as before we let
\[
\Psi(z) := \sup\{u(z): u\in \PSH{\bar\Omega}, u \leq H\}.
\]
Then as before, using Theorem \ref{thm:Edwards}, we see that $\Psi$ is lower semicontinuous. Furthermore, since  $\Psi(z) \leq \Phi_j$, it follows that $\Psi^*\leq \Phi_j^*=\Phi_j$ for all $j$. This means that 
	\[
		\Psi^* \leq \lim_{j \to \infty} \Phi_j= \Phi,
	\]
but since $\Phi$ belongs to the constituting family of $\Psi$, this means that $\Psi = \Phi$, so we see that $\Phi \in \PSH{\bar\Omega}\cap C(\bar\Omega)$.
\end{proof}
\begin{rem}
This theorem could also have been proved using Proposition \ref{prop:hx_ext} and Walsh's theorem, see \cite{Wa}, but we deliberately chose to avoid the use of Walsh's theorem, since its proof heavily relies on the affine structure of $\Cn$. Instead we wanted to show how Jensen measures can be used to study plurisubharmonic envelopes in situations without any affine structure. For a more thorough discussion on this topic, see Wikstr\"om \cite{W2}.
\end{rem}

\section{Acknowledgements}
It is a pleasure to thank professor E. Poletsky of Syracuse University, H. H. Ph\d{a}m of Hanoi National University of Education and P. \AA{}hag and professor U. Cegrell of Ume\aa{} University for inspiring discussions on the theme of this paper.

\end{document}